\documentclass[12pt,reqno]{amsart}
\usepackage{amsmath, amsthm, amssymb}

\advance \topmargin by -\headheight
\advance \topmargin by -\headsep
     
\evensidemargin \oddsidemargin
\newtheorem{theorem}{Theorem}[section]
\newtheorem{lemma}[theorem]{Lemma}

\theoremstyle{definition}

\theoremstyle{remark}

\numberwithin{equation}{section}

\newcommand{\mmod}[1]{\,\,(\text{mod}\,\,#1)}

\def\bfx{{\mathbf x}}
\def\bfy{{\mathbf y}}

\def\calM{{\mathcal M}}

\def\calZ{{\mathcal Z}}

\def\ftil{\widetilde{f}}

\def\dbC{{\mathbb C}}\def\dbN{{\mathbb N}}
\def\dbR{{\mathbb R}}
\def\dbZ{{\mathbb Z}}

\def\grM{{\mathfrak M}}\def\grN{{\mathfrak N}}
\def\grS{{\mathfrak S}}\def\grP{{\mathfrak P}}

\def\grp{{\mathfrak p}}

 \def\grX{{\mathfrak X}}

\def\alp{{\alpha}} 
\def\bet{{\beta}}  
\def\gam{{\gamma}}  \def\Gam{{\Gamma}}
\def\del{{\delta}}

\def\d{{\partial}}
\def\eps{\varepsilon}

\def\le{\leqslant} \def\ge{\geqslant}

\def\d{{\,{\rm d}}}

\begin{document}
\title[Waring's problem]{On Waring's problem: two squares,\\ two cubes and two sixth powers}
\author[T. D. Wooley]{Trevor D. Wooley}
\address{School of Mathematics, University of Bristol, University Walk, Clifton, Bristol BS8 1TW, United Kingdom}
\email{matdw@bristol.ac.uk}
\subjclass[2010]{11P05, 11D85, 11P55}
\keywords{Waring's problem, Hardy-Littlewood method}
\date{}
\begin{abstract} We investigate the number of representations of a large positive integer as the sum of two squares, two positive integral cubes, and two sixth powers, showing that the anticipated asymptotic formula fails 
for at most $O((\log X)^{3+\eps})$ positive integers not exceeding $X$.\end{abstract}
\maketitle

\section{Introduction} The Hardy-Littlewood (circle) method fails to establish the solubility of problems of Waring-type when the sum of the reciprocals of the exponents does not exceed $2$. This well-known consequence of the convexity barrier has been circumvented in very few cases by other devices. A problem that 
fails to be accessible to the circle method by the narrowest of margins is the notorious one of representing integers as sums of two squares, two positive integral cubes, and two sixth powers. The author has recently applied Golubeva's method \cite{Gol2008} to show, subject to the truth of the Generalised Riemann Hypothesis (GRH), that all large integers are thus represented (see \cite[Theorem 1.1]{Woo2012}). However, this work employs representations of special type, and fails to deliver the anticipated asymptotic formula for their total number. The purpose of 
this note is to show that, although the expected asymptotic formula may occasionally fail to hold, the set of such exceptional instances is extraordinarily sparse.\par

Let $R(n)$ denote the number of representations of the integer $n$ in the shape
\begin{equation}\label{1.1}
x_1^2+x_2^2+x_3^3+x_4^3+x_5^6+x_6^6=n,
\end{equation}
with $x_i\in \dbN$ $(1\le i\le 6)$. Also, let
\begin{equation}\label{1.2}
\grS(n)=\sum_{q=1}^\infty \sum_{\substack{a=1\\ (a,q)=1}}^q q^{-6}S_2(q,a)^2S_3(q,a)^2S_6(q,a)^2e(-na/q),
\end{equation}
in which we write
$$S_k(q,a)=\sum_{r=1}^qe(ar^k/q),$$
and as usual $e(z)$ denotes $e^{2\pi i z}$. We refer to a function $\psi(t)$ as being a {\it sedately increasing function} when $\psi(t)$ is a function of a positive variable $t$, increasing monotonically to infinity, and satisfying the condition that when $t$ is large, one has $\psi(t)=O(t^\del)$ for a positive number $\del$ sufficiently small in the ambient context. Finally, write $E(X;\psi)$ for the number of integers $n$ with $1\le n\le X$ such that
\begin{equation}\label{1.3}
\Bigl|R(n)-\tfrac{27}{32}{\sqrt[3]{2}}\Gam(\tfrac{4}{3})^6\grS(n)n\Bigr|>n\psi(n)^{-1}.
\end{equation}
It is worth remarking here that for every integer $n$, one has $1\ll \grS(n)\ll 1$.

\begin{theorem}\label{theorem1.1}
When $\psi(t)$ is a sedately increasing function one has
$$E(X;\psi)\ll \psi(X)^2(\log X)^3.$$
\end{theorem}

It follows that, for each $\eps>0$, the asymptotic formula
\begin{equation}\label{1.4}
R(n)\sim \tfrac{27}{32}\sqrt[3]{2}\Gam(\tfrac{4}{3})^6\grS(n)n
\end{equation}
fails for at most $O((\log X)^{3+\eps})$ of the integers $n$ with $1\le n\le X$, an extremely slim exceptional set indeed. It follows from the final paragraph of \S4, in fact, that if for a large integer $n$ the asymptotic formula (\ref{1.4}) fails, then this formula instead holds for a modification of $n$ in which at most $3$ of its digits are altered.\par

We note that Hooley \cite{Hoo1981a} has confirmed the anticipated asymptotic formula for the number of representations in the analogous problem in which an additional $h$th power is present in (\ref{1.1}). Also, the methods of Vaughan \cite{Vau1980} apply if one is prepared instead to substitute a fifth power for one of the sixth powers. Both of these situations, of course, are not inhibited by the convexity barrier. There are very few instances in which Waring's problem has been resolved in subconvex situations. Subject to natural conditions of congruential type, Gauss \cite{Gau1801} addressed sums of three squares, and Hooley \cite{Hoo1981b} successfully considered sums of two squares and three cubes. Also, work of Golubeva \cite{Gol2008} addresses the mixed problem
$$n=x_1^2+x_2^2+x_3^3+x_4^3+x_5^4+x_6^{16}+x_7^{4k+1}.$$
Most recently of all, subject to the validity of the Elliott-Halberstam conjecture together with GRH, work of Friedlander and the author \cite{FW2012} addresses sums of two squares and three biquadrates.\par

We finish by recording a convention concerning the use of the number $\eps$. Whenever $\eps$ appears in a statement, either implicitly or explicitly, we assert that the statement holds for each $\eps>0$. Note that the ``value'' of $\eps$ may consequently change from statement to statement.

\section{Preparatory manoeuvres} The method that we employ to prove Theorem \ref{theorem1.1} is motivated by our earlier work on sums of one square and five cubes \cite{Woo2002a}. Suppose that $X$ is a large positive number, and let $\psi(t)$ be a sedately increasing function. We denote by $\calZ(X)$ the set of integers $n$ with $X/2<n\le X$ for which the lower bound (\ref{1.3}) holds, and we abbreviate $\text{card}(\calZ(X))$ to $Z$. Write $P_k$ for $[X^{1/k}]$, and define the exponential sum $f_k(\alp)$ by
$$f_k(\alp)=\sum_{1\le x\le P_k}e(\alp x^k).$$
Also, when $Q$ is a positive integer, let $\grM(Q)$ denote the union of the intervals
$$\grM(q,a)=\{ \alp \in [0,1):|q\alp -a|\le QX^{-1}\},$$
with $0\le a\le q\le Q$ and $(a,q)=1$. Then, when $1\le Q\le 2X^{1/2}$, we put $\grN(Q)=\grM(Q)\setminus \grM(Q/2)$. In order to ensure that each $\alp\in [0,1)$ is associated with a uniquely defined arc $\grM(q,a)$, we adopt the convention that when $\alp$ lies in more than one arc $\grM(q,a)\subseteq \grM(Q)$, then it is declared to lie in the arc for which $q$ is least. Finally, let $\nu$ be a sufficiently small positive number, and write $W=X^\nu$. We then take $\grP$ to be the union of the intervals
$$\grP(q,a)=\{\alp \in [0,1): |\alp-a/q|\le WX^{-1}\},$$
with $0\le a\le q\le W$ and $(a,q)=1$, and we set $\grp=[0,1)\setminus \grP$.\par

We initiate proceedings by outlining how the methods of \cite[Chapter 4]{Vau1997} lead to the asymptotic formula
\begin{equation}\label{4.1}
\int_\grP f_2(\alp)^2f_3(\alp)^2f_6(\alp)^2e(-n\alp)\d\alp =\frac{\Gam(\tfrac{3}{2})^2\Gam(\tfrac{4}{3})^2\Gam(\tfrac{7}{6})^2}{\Gam(2)}\grS(n)n+O(n^{1-\tau}),
\end{equation}
for a suitably small positive number $\tau$. Here, the singular series $\grS(n)$ is the one defined in (\ref{1.2}). Recall the identities
$$\Gam(z)\Gam(1-z)=\pi/\sin (\pi z)\quad \text{and}\quad \Gam(z)\Gam(z+\tfrac{1}{2})=2^{1-2z}\sqrt{\pi}\Gam(2z).$$
Then one finds that
$$3\Gam(\tfrac{4}{3})\Gam(\tfrac{2}{3})=\pi/\sin (\pi/3)\quad \text{and}\quad \Gam(\tfrac{2}{3})\Gam(\tfrac{7}{6})=2^{1-4/3}\sqrt{\pi}\Gam(\tfrac{4}{3}),$$
and hence
$$\frac{\Gam(\tfrac{3}{2})^2\Gam(\tfrac{4}{3})^2\Gam(\tfrac{7}{6})^2}{\Gam(2)}=\frac{\tfrac{1}{4}\pi \Gam(\tfrac{4}{3})^2(2^{-1/3}\sqrt{\pi}\Gam(\tfrac{4}{3}))^2}{\Gam(\tfrac{2}{3})^2}=\frac{2^{-8/3}\pi^2\Gam(\tfrac{4}{3})^6}{(2\pi/3\sqrt{3})^2}=\tfrac{27}{32}\sqrt[3]{2}\Gam(\tfrac{4}{3})^6.$$

\par Write
$$v_k(\bet)=\int_0^{P_k}e(\bet \gam^k)\d\gam ,$$
and for the time being define $f_k^*(\alp)$ for $\alp\in \grP(q,a)\subseteq \grP$ by putting
\begin{equation}\label{4.2}
f_k^*(\alp)=q^{-1}S_k(q,a)v_k(\alp-a/q).
\end{equation}
Then it follows from \cite[Theorem 4.1]{Vau1997} that whenever $\alp \in \grP(q,a)\subseteq \grP$, one has
$$f_k(\alp)-f_k^*(\alp)\ll q^{1/2+\eps}\ll W^{1/2+\eps}.$$
The measure of $\grP$ is $O(W^3X^{-1})$, and hence we deduce that
\begin{align*}
\int_\grP f_2(\alp)^2f_3(\alp)^2f_6(\alp)^2e(-n\alp)\d\alp -&\int_\grP f_2^*(\alp)^2f_3^*(\alp)^2f_6^*(\alp)^2e(-n\alp)\d\alp \\
&\ll W^4X^{5/6}.
\end{align*}
A routine computation leads from (\ref{4.2}) to the formula
$$\int_\grP f_2^*(\alp)^2f_3^*(\alp)^2f_6^*(\alp)^2e(-n\alp)\d\alp =\grS(n;W)J(n;W),$$
where
\begin{equation}\label{4.3}
\grS(n;W)=\sum_{1\le q\le W}A(q;n),
\end{equation}
in which we have written
\begin{equation}\label{4.4}
A(q;n)=\sum^q_{\substack{a=1\\ (a,q)=1}}q^{-6}S_2(q,a)^2S_3(q,a)^2S_6(q,a)^2e(-na/q),
\end{equation}
and
$$J(n;W)=\int_{-WX^{-1}}^{WX^{-1}}v_2(\bet)^2v_3(\bet)^2v_6(\bet)^2e(-\bet n)\d\bet.$$
Thus we may conclude at this stage of the argument that
\begin{equation}\label{4.5}
\int_\grP f_2(\alp)^2f_3(\alp)^2f_6(\alp)^2e(-n\alp)\d\alp =\grS(n;W)J(n;W)+O(X^{1-\tau}),
\end{equation}
for a suitable positive number $\tau$.\par

We next recall from \cite[Theorem 7.3]{Vau1997} that
$$v_k(\bet)\ll P_k(1+|\bet|X)^{-1/k},$$
whence the singular integral $J(n;W)$ converges absolutely as $W\rightarrow \infty$. The discussion concluding \cite[Chapter 4]{Dav1963} consequently delivers the relation
\begin{equation}\label{4.6}
J(n;W)=\frac{\Gam(\tfrac{3}{2})^2\Gam(\tfrac{4}{3})^2\Gam(\tfrac{7}{6})^2}{\Gam(2)}n+O(n^{1-\tau}).
\end{equation}

\par In order to analyse the truncated singular series $\grS(n;W)$, we begin by defining the multiplicative function $w_k(q)$ by taking
$$w_k(p^{uk+v})=\begin{cases} kp^{-u-1/2},&\text{when $u\ge 0$ and $v=1$,}\\
p^{-u-1},&\text{when $u\ge 0$ and $2\le v\le k$.}\end{cases}$$
Then according to \cite[Lemma 3]{Vau1986}, whenever $a\in \dbZ$ and $q\in \dbN$ satisfy $(a,q)=1$, one has $q^{-1}S_k(q,a)\ll w_k(q)$. On employing this estimate within (\ref{4.4}), we see that when $q=p^h$, then
$$q^{1/2}A(q;n)\ll q^{3/2}w_2(q)^2w_3(q)^2w_6(q)^2\ll p^{-\max\{3/2,h/2\}},$$
and thus
$$\sum_{h=1}^\infty p^{h/2}|A(p^h;n)|\ll p^{-3/2}+\sum_{h=4}^\infty p^{-h/2}\ll p^{-3/2}.$$
By applying the multiplicativity of the function $A(q;n)$, we consequently deduce that for a suitable positive number $B$, one has
\begin{align*}
\sum_{q=1}^\infty (q/W)^{1/2}|A(q;n)|&=W^{-1/2}\prod_p \Bigl( 1+\sum_{h=1}^\infty p^{h/2}|A(p^h;n)|\Bigr) \\
&\ll W^{-1/2}\prod_p (1+Bp^{-3/2}) \ll W^{-1/2}.
\end{align*}
In this way, we find that
$$\sum_{q>W}|A(q;n)|\ll W^{-1/2},$$
so that the singular series $\grS(n)$ defined in (\ref{1.2}) converges absolutely. In particular, we find from (\ref{4.3}) that one has
$$\grS(n)-\grS(n;W)\ll W^{-1/2}.$$
On recalling (\ref{4.5}) and (\ref{4.6}), we confirm that the formula (\ref{4.1}) does indeed hold.\par

We pause here to note that the standard theory of singular series in Waring's problem (see \cite[Chapter 4]{Vau1997}) shows that $\grS(n)\gg 1$ provided only that the congruence
$$x_1^2+x_2^2+x_3^3+x_4^3+x_5^6+x_6^6\equiv n\mmod{q}$$
is soluble for every natural number $q$. That such is indeed the case follows directly by applying the Cauchy-Davenport Lemma (see \cite[Lemma 2.14]{Vau1997}). This confirms the remark preceeding the statement of Theorem \ref{theorem1.1} above to the effect that $1\ll \grS(n)\ll 1$ for every integer $n$.\par

We are now equipped to examine the exceptional set $\calZ(X)$. For each integer $n\in \calZ(X)$, it follows from (\ref{1.3}) via orthogonality that
$$\Bigl| \int_0^1 f_2(\alp)^2f_3(\alp)^2f_6(\alp)^2e(-n\alp)\d\alp -\tfrac{27}{32}\sqrt[3]{2}\Gam(\tfrac{4}{3})^6\grS(n)n\Bigr| >\tfrac{1}{2}X\psi(X)^{-1}.$$
By employing the relation (\ref{4.1}), therefore, we obtain the lower bound
$$\Bigl| \int_\grp f_2(\alp)^2f_3(\alp)^2f_6(\alp)^2e(-n\alp)\d\alp \Bigr| >\tfrac{1}{4} X\psi(X)^{-1},$$
and from this we see that
\begin{equation}\label{4.7}
\sum_{n\in \calZ(X)}\Bigl| \int_\grp f_2(\alp)^2f_3(\alp)^2f_6(\alp)^2e(-n\alp)\d\alp \Bigr| \gg ZX\psi(X)^{-1}.
\end{equation}

\par For each integer $n$ with $n\in \calZ(X)$, there exists a complex number $\eta_n$, with $|\eta_n|=1$, satisfying the condition that
$$\Bigl| \int_\grp f_2(\alp)^2f_3(\alp)^2f_6(\alp)^2e(-n\alp)\d\alp \Bigr| =\eta_n\int_\grp f_2(\alp)^2f_3(\alp)^2f_6(\alp)^2e(-n\alp)\d\alp .$$
Define the exponential sums $K(\alp)$ and $K^*(\alp)$ by putting
$$K(\alp)=\sum_{n\in \calZ(X)}\eta_ne(-n\alp)\quad \text{and}\quad K^*(\alp)=\sum_{n\in \calZ(X)}e(-n\alp).$$
Then it follows from (\ref{4.7}) that
\begin{align}
ZX\psi(X)^{-1}&\ll \sum_{n\in \calZ(X)}\eta_n\int_\grp f_2(\alp)^2f_3(\alp)^2f_6(\alp)^2e(-n\alp)\d\alp \notag\\
&\le \int_\grp |f_2(\alp)^2f_3(\alp)^2f_6(\alp)^2K(\alp)|\d\alp .\label{4.8}
\end{align}
This relation is the starting point for our estimation of $Z$, which we achieve by bounding the integral on the right hand side of (\ref{4.8}).

\section{Auxiliary estimates} Before making further progress, we pause to derive some auxiliary mean value estimates of use in our subsequent argument, and these we collect together in the present section. In this and the next section, we write $L=L(X)$ for $\log X$. We begin by supplying a sharp version of a mean value involving both $f_3(\alp)$ and $f_6(\alp)$.

\begin{lemma}\label{lemma5.1} One has
\begin{equation}\label{5.1}
\int_0^1|f_3(\alp)^2f_6(\alp)^4|\d\alp \le 3X^{2/3}+O(X^{23/36+\eps}).
\end{equation}
\end{lemma}

\begin{proof} By orthogonality, the integral $I_1$ on the left hand side of (\ref{5.1}) counts the number of integral solutions of the equation
\begin{equation}\label{5.2}
x_1^3-x_2^3=y_1^6+y_2^6-y_3^6-y_4^6,
\end{equation}
with $1\le x_1,x_2\le P_3$ and $1\le y_i\le P_6$ $(1\le i\le 4)$. The number of solutions of (\ref{5.2}) counted by $I_1$ with $x_1=x_2$ is at most $P_3I_2$, where $I_2$ denotes the number of integral solutions of the equation
$$y_1^6+y_2^6=y_3^6+y_4^6,$$
with $1\le y_i\le P_6$ $(1\le i\le 4)$. It was shown by Greaves \cite{Gre1994}, Skinner and Wooley \cite{SW1995} and Hooley \cite{Hoo1996}, in close chronological proximity, that $I_2=2P_6^2+O(P_6^{2-\del})$ for some $\del>0$. The last named source obtains such a conclusion for any 
$\del<\tfrac{1}{3}$. The most recent results can be found in Heath-Brown \cite[Theorem 11]{HB2002}. On the other hand, the solutions of (\ref{5.2}) counted by $I_1$ with $x_1\ne x_2$ are of two types. There are the solutions $\bfx$, $\bfy$ in which the integer
\begin{equation}\label{5.3}
y_1^6+y_2^6-y_3^6-y_4^6
\end{equation}
has precisely one representation in the form $x_1^3-x_2^3$, of which there are at most $P_6^4$. When the expression (\ref{5.3}) has more than one representation in the shape $x_1^3-x_2^3$, meanwhile, we must proceed in a more elaborate fashion.\par

It follows from the work of Heath-Brown \cite{HB1997}, as explained following equation (9.9) of Wooley \cite{Woo2002b}, that there are $O(P_3^{4/3+\eps})$ integers $m$ having two or more representations in the shape $m=x_1^3-x_2^3$, with $1\le x_1,x_2\le P_3$. Write $\grX$ for the set of such integers $m$, and $\rho(n)$ for the number of representations of an integer $n$ in the shape $n=y_1^6+y_2^6-y_3^6-y_4^6$, with $1\le y_i\le P_6$ $(1\le i\le 4)$. Then by Cauchy's inequality, the number of solutions of (\ref{5.2}), in which (\ref{5.3}) has more than one representation in the shape $x_1^3-x_2^3$, is
$$\sum_{m\in \grX}\rho(m)\le \Bigl( \sum_{m\in \grX}1\Bigr)^{1/2}\Bigl( \sum_{m\in \grX}\rho(m)^2\Bigr)^{1/2}\ll (P_3^{4/3+\eps})^{1/2}\Bigl( \int_0^1|f_6(\alp)|^8\d\alp \Bigr)^{1/2}.$$
By Hua's Lemma (see \cite[Lemma 2.5]{Vau1997}), one has
$$\int_0^1|f_6(\alp)|^8\d\alp \ll P_6^{5+\eps},$$
and thus we conclude that
$$\sum_{m\in \grX}\rho(m)\ll (P_3^{4/3+\eps})^{1/2}(P_6^{5+\eps})^{1/2}\ll X^{23/36+\eps}.$$

\par Combining all of the above contributions, therefore, we conclude that
$$I_1\le 2P_3P_6^2+P_6^4+O(X^{23/36+\eps})=3X^{2/3}+O(X^{23/36+\eps}).$$
This completes the proof of the lemma.
\end{proof}

\begin{lemma}\label{lemma5.2} One has
$$\int_0^1|f_3(\alp)^2K(\alp)^2|\d\alp \ll P_3Z+P_3^\eps Z^2.$$
\end{lemma}

\begin{proof} The desired conclusion is immediate from \cite[Lemma 2.1]{Woo2003} and its proof.
\end{proof}

We finish this section by recording a variant of Br\"udern's pruning lemma suitable for our purposes.

\begin{lemma}\label{lemma5.3} Let $N$, $Q$ be real numbers with $Q\le N$. For $0\le a\le q\le Q$ and $(a,q)=1$, let $\calM(q,a)$ denote an interval contained in $[a/q-\tfrac{1}{2},a/q+\tfrac{1}{2}]$. Write $\calM$ for the union of all $\calM(q,a)$. Let $\gam<1$ be a positive number, and let $G:\calM\rightarrow \dbC$ be a function which for $\alp \in \calM(q,a)$ satisfies
$$G(\alp)\ll (q+N|q\alp -a|)^{-1-\gam}.$$
Furthermore, let $\Psi:\dbR\rightarrow [0,\infty)$ be a function with a Fourier expansion
$$\Psi(\alp)=\sum_{|h|\le H}\psi_he(\alp h)$$
such that $\log H\ll \log N$. Then
$$\int_\calM G(\alp)\Psi(\alp)\d\alp \ll_\gam N^{-1}\Bigl(Q^{1-\gam}\psi_0+N^{\eps}\sum_{h\ne 0}|\psi_h|\Bigr).$$
\end{lemma}

\begin{proof} Following the argument of the proof of \cite[Lemma 2]{Bru1988}, mutatis mutandis, we see that
$$\int_\calM G(\alp)\Psi(\alp)\d\alp \ll \sum_{q\le Q}\sum^q_{\substack{a=1\\ (a,q)=1}}\sum_{|h|\le H}q^{-1-\gam}\psi_h\rho_he(ah/q),$$
where
\begin{equation}\label{5.4}
\rho_h=\int_{-1/2}^{1/2}(1+N|\alp|)^{-1-\gam}e(\alp h)\d\alp \ll_\gam N^{-1}.
\end{equation}
Hence, just as in the analogous argument of \cite[Lemma 2]{Bru1988}, we obtain
\begin{align*}
\int_\calM G(\alp)\Psi(\alp)\d\alp &\ll \sum_{q\le Q}q^{-\gam}\psi_0\rho_0+\sum_{h\ne 0}|\psi_h\rho_h|\sum_{q\le Q}\sum_{d|(q,h)}dq^{-1-\gam}\\
&\ll_\gam Q^{1-\gam}\psi_0\rho_0+N^\eps \sum_{h\ne 0}|\psi_h\rho_h|.
\end{align*}
Recalling the upper bound (\ref{5.4}), the desired conclusion follows at once.
\end{proof}

\section{The core analysis} We initiate the main thrust of our work on the integral on the right hand side of (\ref{4.8}) by introducing some auxiliary mean values that include major arc majorants to $f_k(\alp)$ $(k=2,3)$. Let $\del$ be a sufficiently small positive number. Define the function $\ftil_2(\alp)$ for $\alp \in \grM(q,a)\subseteq \grM(Q)$ by taking
\begin{equation}\label{6.1}
\ftil_2(\alp)=P_2(q+P_2^2|q\alp-a|)^{-1/2},
\end{equation}
and then put
\begin{equation}\label{6.2}
T_1(Q)=\int_{\grN(Q)}\ftil_2(\alp)^2|f_3(\alp)^2f_6(\alp)^2K(\alp)|\d\alp .
\end{equation}
In addition, define the function $f_3^*(\alp)$ for $\alp \in \grM(q,a)\subseteq \grM(Q)$ by taking
\begin{equation}\label{6.3}
f_3^*(\alp)=q^{-1}S_3(q,a)v_3(\alp-a/q), 
\end{equation}
and then put
\begin{equation}\label{6.4}
T_2(Q)=\int_{\grN(Q)}\ftil_2(\alp)^2|f_3^*(\alp)^2f_6(\alp)^2K(\alp)|\d\alp .
\end{equation}

\begin{lemma}\label{lemma6.1}
When $1\le Q\le X^{1/2}$, one has
$$T_1(Q)\ll XZ^{1/2}+X^{7/6+\eps}Q^{\del-1/2}Z.$$
\end{lemma}

\begin{proof} Let $\gam$ be a small positive number. Then by applying H\"older's inequality to (\ref{6.2}), one obtains the bound
\begin{align}
T_1(Q)\le &\,\Bigl(\sup_{\alp \in \grN(Q)}\ftil_2(\alp)\Bigr)^{1-\gam}\Bigl( \int_0^1 |f_3(\alp)^2f_6(\alp)^4|\d\alp \Bigr)^{1/2}\notag \\
&\times \Bigl( \int_{\grN(Q)}\ftil_2(\alp)^{2+2\gam}|f_3(\alp)^2K(\alp)^2|\d\alp \Bigr)^{1/2}.\label{6.5}
\end{align}
Making use of (\ref{6.1}) and applying Lemma \ref{lemma5.3}, we see that
\begin{align*}
\int_{\grN(Q)}&\ftil_2(\alp)^{2+2\gam}|f_3(\alp)^2K(\alp)^2|\d\alp \\
&\ll P_2^{2+2\gam}X^{-1}\Bigl( Q^{1-\gam}\int_0^1|f_3(\alp)^2K(\alp)^2|\d\alp +X^\eps f_3(0)^2K^*(0)^2\Bigr) .
\end{align*}
Then as a consequence of Lemma \ref{lemma5.2}, one finds that
\begin{align*}
\int_{\grN(Q)}\ftil_2(\alp)^{2+2\gam}&|f_3(\alp)^2K(\alp)^2|\d\alp \\
&\ll P_2^{2+2\gam}X^{-1}\left( Q^{1-\gam}(P_3Z+P_3^\eps Z^2)+P_3^{2+\eps}Z^2\right) \\
&\ll P_2^{2\gam}(P_3Q^{1-\gam}Z+P_3^{2+\eps}Z^2).
\end{align*}
On recalling (\ref{6.1}) and Lemma \ref{lemma5.1}, we therefore conclude from (\ref{6.5}) that
\begin{align*}
T_1(Q)&\ll (P_2Q^{-1/2})^{1-\gam}(X^{2/3})^{1/2}P_2^\gam (P_3Q^{1-\gam}Z+P_3^{2+\eps}Z^2)^{1/2}\\
&\ll P_2X^{1/3}P_3^{1/2}Z^{1/2}+P_2X^{1/3+\eps}P_3Q^{(\gam-1)/2}Z.
\end{align*}
The conclusion of the lemma follows on recalling that $P_k=X^{1/k}$.
\end{proof}

\begin{lemma}\label{lemma6.2}
When $1\le Q\le X^{1/2}$, one has
$$T_2(Q)\ll X^{5/6+\eps}Q^{1/3}Z+X^{1+\eps}Q^{-2/3}Z.$$
\end{lemma}

\begin{proof} On making use of the definition (\ref{6.3}), we find from \cite[Theorems 4.2 and 7.3]{Vau1997} that when $\alp \in \grM(q,a)\subseteq \grM(Q)$ one has
$$f_3^*(\alp)\ll P_3(q+P_3^3|q\alp-a|)^{-1/3},$$
whence
$$\sup_{\alp\in \grN(Q)}|f_3^*(\alp)|\ll P_3Q^{-1/3}.$$
On recalling (\ref{6.1}), we therefore deduce from \cite[Lemma 2]{Bru1988} and (\ref{6.4}) that whenever $\gam$ is a sufficiently small positive number, one has
\begin{align*}
T_2(Q)&\ll ( \sup_{\alp \in \grN(Q)}|f_3^*(\alp)|)^2K^*(0)\int_{\grN(Q)}\ftil_2(\alp)^2|f_6(\alp)|^2\d\alp \\
&\ll X^\eps P_3^2Q^{-2/3}Z\Bigl( Q\int_0^1|f_6(\alp)|^2\d\alp +P_6^2\Bigr)\\
&\ll X^\eps P_3^2Q^{-2/3}Z(QP_6+P_6^2).
\end{align*}
The conclusion of the lemma now follows. 
\end{proof}

Write
\begin{equation}\label{6.6}
T_0(Q)=\int_{\grN(Q)}|f_2(\alp)^2f_3(\alp)^2f_6(\alp)^2K(\alp)|\d\alp .
\end{equation}
According to \cite[Theorem 4]{Vau2006}, when $\alp \in \grM(q,a)\subseteq \grM(Q)$ and $1\le Q\le 2P_2$, one has
\begin{align*}
f_2(\alp)&\ll P_2(q+P_2^2|q\alp-a|)^{-1/2}+(q+P_2^2|q\alp-a|)^{1/2}\\
&\ll P_2(q+P_2^2|q\alp-a|)^{-1/2}.
\end{align*}
Consequently, when $\alp\in \grN(Q)$ and $1\le Q\le X^{1/2}$, one has
\begin{equation}\label{6.7}
f_2(\alp)\ll \ftil_2(\alp).
\end{equation}
It therefore follows from (\ref{6.6}) and Lemma \ref{lemma6.1} that whenever $X^{1/3+\del}\le Q\le X^{1/2}$, one has
\begin{align}
T_0(Q)&\ll \int_{\grN(Q)}\ftil_2(\alp)^2|f_3(\alp)^2f_6(\alp)^2K(\alp)|\d\alp \notag \\
&\ll XZ^{1/2}+X^{7/6+\eps}Q^{\del-1/2}Z\notag \\
&\ll XZ^{1/2}+X^{1-\del^2}Z.\label{6.8}
\end{align}

\par Next, from \cite[Theorem 4.1]{Vau1997}, one finds that when $\alp\in \grN(Q)$ one has
$$f_3(\alp)-f_3^*(\alp)\ll Q^{1/2+\eps},$$
whence
\begin{align*}
\int_{\grN(Q)}\ftil_2(\alp)^2|f_3(\alp)^2f_6(\alp)^2K(\alp)|\d\alp \ll &\, \int_{\grN(Q)}\ftil_2(\alp)^2|f_3^*(\alp)^2f_6(\alp)^2K(\alp)|\d\alp \\
&\, +Q^{1+\eps}K^*(0)\int_{\grN(Q)}\ftil_2(\alp)^2|f_6(\alp)|^2\d\alp .
\end{align*}
On recalling (\ref{6.1}), (\ref{6.4}) and \cite[Lemma 2]{Bru1988}, therefore, we see as in the proof of Lemma \ref{lemma6.2} that
$$\int_{\grN(Q)}\ftil_2(\alp)^2|f_3(\alp)^2f_6(\alp)^2K(\alp)|\d\alp \ll T_2(Q)+X^\eps QZ(QP_6+P_6^2).$$
Then on making use of Lemma \ref{lemma6.2}, we conclude from (\ref{6.7}) that whenever $1\le Q\le X^{1/2}$, one has
\begin{align*}
T_0(Q)&\ll \int_{\grN(Q)}\ftil_2(\alp)^2|f_3(\alp)^2f_6(\alp)^2K(\alp)|\d\alp \\
&\ll X^\eps Z(Q^2X^{1/6}+QX^{1/3}+Q^{1/3}X^{5/6}+Q^{-2/3}X).
\end{align*}
Consequently, when $X^\nu \le Q\le X^{1/3+\del}$ one deduces that
\begin{equation}\label{6.9}
T_0(Q)\ll X^\eps Z(X^{17/18+\del}+X^{1-2\nu/3}).
\end{equation}

\par We now cover $\grp$ by dyadic sets of arcs $\grN(Q)$ with $Q=X^\nu, 2X^\nu,\ldots $ and $Q\le X^{1/2}$, as is possible as a consequence of Dirichlet's theorem on Diophantine approximation. Thus we see from (\ref{6.8}) and (\ref{6.9}) that
$$\int_\grp |f_2(\alp)^2f_3(\alp)^2f_6(\alp)^2K(\alp)|\d\alp\ll L(XZ^{1/2}+X^{1-\del^2}Z+X^{1-\nu/2}Z).
$$
On substituting into (\ref{4.8}), we therefore conclude that
$$ZX\psi(X)^{-1}\ll LXZ^{1/2}+X^{1-\del^3}Z,$$
whence $Z\ll L^2\psi(X)^2$. Summing over dyadic intervals to cover the set of integers $[1,X]\cap \dbZ$, we therefore arrive at the estimate
$$E(X;\psi)\le \sum_{2^j\le X}\text{card}(\calZ(2^{j+1}))\ll L^3\psi(X)^2.$$
This completes the proof of Theorem \ref{theorem1.1}.

\bibliographystyle{amsbracket}
\providecommand{\bysame}{\leavevmode\hbox to3em{\hrulefill}\thinspace}

\end{document}